\documentclass{article}
\usepackage[utf8]{inputenc}

\title{The Group of Boundary Fixing Homeomorphisms of the Disc is Not Left-Orderable}
\author{James Hyde}

\usepackage{graphicx}
\usepackage{thmtools,thm-restate}
\usepackage{hyperref}

\usepackage[a4paper]{geometry}
\usepackage{mathrsfs,amsmath, tikz,verbatim, wasysym,color, amsfonts, amssymb, amsthm}

\usepackage{scalerel,stackengine}
\stackMath
\newcommand\reallywidehat[1]{%
\savestack{\tmpbox}{\stretchto{%
  \scaleto{%
    \scalerel*[\widthof{\ensuremath{#1}}]{\kern-.6pt\bigwedge\kern-.6pt}%
    {\rule[-\textheight/2]{1ex}{\textheight}}
  }{\textheight}%
}{0.5ex}}%
\stackon[1pt]{#1}{\tmpbox}%
}

\newtheorem{thm}{Theorem}[section]

\newtheorem{lem}[thm]{Lemma}

{\theoremstyle{definition}

\newtheorem{defn}[thm]{Definition}

}

\newtheorem*{thm*}{Theorem}
\newtheorem*{defn*}{Definition}
\newtheorem*{cor*}{Corollary}

\newcommand{\Ss}[2]{\left\{ #1 \hspace{1mm} \middle | \hspace{1mm} #2 \right\}}

\newcommand{\Rr}{\mathbb{R}}

\newcommand{\homeo}{\operatorname{Homeo}}

\newcommand{\abbor}{\mathrel{\ll}}

\newcommand{\HD}{\homeo (I^2;\delta I^2)}
\newcommand{\HI}{\homeo (I;\delta I)}

\newcommand{\An}[1]{\left\langle #1 \right\rangle}

\newcommand{\idnt}{\mathbf{1}}

\begin{document}

\maketitle

\begin{abstract}
A left-order on a group $G$ is a total order $<$ on $G$ such that for any $f$, $g$ and $h$ in $G$ we have $f < g \Leftrightarrow hf < hg$. We construct a finitely generated subgroup $H$ of $\HD$, the group of those homeomorphisms of the disc that fix the boundary pointwise, and show $H$ does not admit a left-order. Since any left-order on $\HD$ would restrict to a left-order on $H$ this shows that $\HD$ does not admit a left-order. Since $\HI$ admits a left-order it follows that neither $H$ nor $\HD$ embed in $\HI$.
\end{abstract}

\section{Introduction}

\begin{defn}
Let $G$ be a group. A total order $<$ on $G$ is a \emph{left-order} if for any $f$, $g$ and $h$ in $G$ we have $f < g \Leftrightarrow hf < hg$.
\end{defn}

Right-orders may be defined analogously to left-orders but we deal only with left-orders. We will write $\HD$ for the group of those homeomorphisms of the disc that fix the boundary pointwise.

Whether $\HD$ admits a left-order has been asked
by Calegari, Clay, Deroin, Navas, Rivas and Rolfsen in \cite{Calegari}, \cite{Navas1}, \cite{CalegariRolfsen}, \cite{ClayRolfsen} and \cite{Navas2}. Also Bergman's question $17.20$ in \cite{Kourovka} is equivalent to this question by Proposition $8.8$ of \cite{ClayRolfsen}.

We resolve this question with Theorem \ref{thm:main2} below.

\begin{thm}\label{thm:main2}
If $M$ is a manifold of dimension $n \geq 2$, possibly with boundary, and $S$ is a proper closed subset of $M$ then $\homeo(M;S)$, the group of those homeomorphisms of $M$ that pointwise stabilise $S$, does not admit a left-order.
\end{thm}
\begin{proof}
We may choose a non-empty open subset $U$ of $M$, disjoint from $S$, and homeomorphic to $I^n$. Since $\homeo(I^n;\delta I^n)$ is isomorphic to $\homeo(U;\delta U)$ which embeds in $\homeo(M;S)$ it follows from Proposition $8.8$ of \cite{ClayRolfsen} that $\HD$ embeds in $\homeo(M;S)$. Since the group $H$, defined in Section \ref{sec:disc}, embeds in $\HD$ and by Theorem \ref{thm:main} below does not admit a left-order it follows that $\homeo(M;S)$ does not admit a left-order.
\end{proof}

Since $\HI$ admits a left-order it follows from Theorem \ref{thm:main2} that if $M$ is a manifold of dimension at least $2$, possibly with boundary, then the group of those homeomorphisms of $M$ that fix the boundary pointwise does not embed in $\HI$.

For $n \geq 1$ the group of boundary fixing piecewise linear homeomorphisms of the $n$ dimensional disc is left-orderable (see Theorem $8.6$ of \cite{ClayRolfsen}). The group of boundary fixing diffeomorphisms of the disc is also left-orderable (see \cite{Calegari}).

Beyond basic definitions we will not go into the detail of the theory of left-orderable groups. Instead we direct the reader to the book Groups, Orders and Dynamics \cite{Navas1}. We note that left-ordered groups are torsion-free. 

In Section \ref{sec:Gen} we discuss left-ordered groups in general. In Section \ref{sec:disc} we construct the group $H$ and apply the work of Section \ref{sec:Gen} to show $H$ does not admit a left-order.

\subsection*{Acknowledgements}
I thank Yash Lodha for introducing me to this problem and thank George Bergman and Yash Lodha for helpful comments.

\section{General Approach} \label{sec:Gen}

In order to facilitate the application of the argument of this document to other groups we split the argument between two sections. The work of this section applies to groups in general and the work of the next section applies to $\HD$ specifically. We write $\idnt$ for the identity of a group. For $g$ and $h$ elements of a group with left-order $<$ we write $h > g$ to indicate $g < h$ and write $g \leq h$ to indicate that either $g < h$ or $g=h$.

\begin{lem} \label{lem:list}
Let $G$ be a left-ordered group with left-order $<$. Let $u$, $v$ and $t$ in $G$ be such that $ut = tu$ and $vt =tv$. Let $m$, $n$, $n_1$ and $n_2$ be integers. Then:
\begin{enumerate}
\item \label{point:inv} $u < t^m \iff t^{-m} < u^{-1}$ (the same holds with $>$ in place of $<$).
\item \label{point:sum} $u < t^m$ and $v < t^n$ implies $uv < t^{m+n}$ (the same implication holds with $>$ in place of $<$).
\item \label{point:con} $t^{m-1} < u < t^{m}$ and $t^{n_1} < v < t^{n_2}$ implies $t^{m-2} < u^v < t^{m+1}$.
\item \label{point:bon} $u^v = u^{-1}$ and $t^{n_1} < v < t^{n_2}$ and $\idnt < t$ implies $t^{-1} < u < t$.
\end{enumerate}
\end{lem}
\begin{proof}
With the same numbering as in the statement of the lemma.
\begin{enumerate}

\item We have $u < t^m \iff u^{-1}t^{-m}u < u^{-1}t^{-m}t^m \iff t^{-m} <u^{-1}$ since $ut = tu$.

\item $uv < ut^n = t^nu < t^nt^m = t^{m+n}$.

\item Since $t^{n_1} < v < t^{n_2}$ it follows that there exists an integer $n_0$ such that $t^{n_0} \leq v < t^{{n_0}+1}$. By Point \ref{point:inv} we have $t^{-(n_0+1)} < v^{-1} \leq t^{-n_0}$. Therefore using Point \ref{point:sum} it follows that $t^{m-2} = t^{-(n_0+1)}t^{m-1}t^{n_0} < v^{-1}uv < t^{-n_0}t^mt^{n_0+1} = t^{m+1}$.


\item Since $t^{n_1} < v < t^{n_2}$ it follows that there exists an integer $n_0$ such that $t^{n_0} \leq v < t^{{n_0}+1}$.

If $u < t^{-1}$ then by Point \ref{point:sum} we have $v = uvu < t^{-1}t^{n_0+1}t^{-1} = t^{n_0-1} < t^{n_0}$ a contradiction.

If $t < u$ then by Point \ref{point:sum} we have $t^{n_0+1} < t^{n_0+2} = tt^{n_0}t < uvu = v$ also a contradiction.



\end{enumerate}
\end{proof}

For $g$ an element of a left-ordered group we will write $|g|$ for the greater of $g^{-1}$ and $g$.


\begin{lem} \label{lem:gen}
Let $a$, $b$, $c$ and $d$ be non-identity elements of a left-ordered group $G$ with left-order $<$. Assume $ab = ba$ and $bc = cb$ and $bd = db$ and $c^{(a^3)} = c^{-1}$ and $d^{(a^3)} = d^{-1}$ and $|a| < |b|$. Then $|c^d c^{da} c^{da^{2}} c^{da^{3}} c^{da^{4}}  c^{da^{5}}| < |b^{12}|$.
\end{lem}
\begin{proof}
By symmetry we may assume $\idnt < b$. Therefore $b^{-1} < a < b$. By Point \ref{point:sum} and \ref{point:bon} of Lemma \ref{lem:list} we have $b^{-1} < c < b$ and $b^{-1} < d < b$.

By Point \ref{point:con} of Lemma \ref{lem:list} for each $g \in \{c^d, c^{da}, c^{da^2}, c^{da^3}, c^{da^4}, c^{da^5}\}$ we have $b^{-2} < g < b^2$. Therefore by Point \ref{point:sum} of Lemma \ref{lem:list} we have $b^{-12} < c^d c^{da} c^{da^2} c^{da^3} c^{da^4} c^{da^5} < b^{12}$.
\end{proof}

\section{The Group of Boundary Fixing Homeomorphisms of the Disc} \label{sec:disc}
Fix $\rho$ the homeomorphism from the unit disc to the plane defined, in polar coordinates, by the rule $\rho:(r,\theta) \mapsto (r/(1-r),\theta)$.
Let $f$ be the function from the group of boundary fixing homeomorphisms of the disc to the group of homeomorphisms of the plane defined by the rule $f:\gamma \mapsto \rho^{-1}\gamma\rho$. Since $\rho$ is a homeomorphism it follows that $f$ is an embedding.

Instead of working in the group of boundary fixing homeomorphisms of the disc we work in the image of $f$. For the rest of the document we work with Cartesian coordinates only.

Fix $\alpha:\Rr^2 \to \Rr^2$ defined by the rule $\alpha:(x,y) \mapsto (x+(1/6),y)$. Similarly, fix $\beta:\Rr^2 \to \Rr^2$ defined by the rule $\beta:(x,y) \mapsto (x,y+(1/6))$. Note $\alpha$ and $\beta$ are in the image of $f$.

Define a continuous piecewise linear function $\gamma_0:\Rr \to [-3,3]$ by linearly interpolating from the rules: for each integer $n$
\begin{itemize}
\item $(n)\gamma_0 = 3$, and
\item $(n+(1/2))\gamma_0 = -3$.
\end{itemize}

Fix $\gamma: \Rr^2 \to \Rr^2$ defined by the rule $\gamma:(x,y) \mapsto (x,x\gamma_0 +y)$. Note $\gamma$ is in the image of $f$.

Define a continuous piecewise linear bijection $\delta_0:\Rr \to \Rr$ by linearly interpolating from the rules: for each integer $n$
\begin{itemize}
\item $(n+(1/3))\delta_0 = n+(1/6)$, and
\item $(n+(2/3))\delta_0 = n+(5/6)$.
\end{itemize}

Fix $\delta:\Rr^2 \to \Rr^2$ defined by the rule $\delta:(x,y) \mapsto (x\delta_0,y)$. Note $\delta$ is in the image of $f$.

For each $x$ in $\Rr$ note any element of $\An{\alpha, \beta, \gamma, \delta}$ maps the vertical line $\Ss{(x,y)}{y \in \Rr}$ to another vertical line by an isometry.

For $g$ and $h$ elements of a group $G$ we will write $g \abbor h$ in $G$ if for every left-order $<$ on $G$ we have $|g| < |h|$.

\begin{lem}\label{lem:CM}
$\beta \abbor \alpha$ in $\An{\alpha,\beta,\gamma,\delta}$.
\end{lem}
\begin{proof}
Assume $<$ is a left-order on $\An{\alpha,\beta,\gamma,\delta}$. By inspection $\alpha\beta = \beta\alpha$ and $\beta\gamma = \gamma\beta$ and $\beta\delta = \delta\beta$ and $\gamma^{(\alpha^3)} = \gamma^{-1}$ and $\delta^{(\alpha^3)} = \delta^{-1}$.

It is sufficient to show
$\gamma^\delta
\gamma^{\delta\alpha}
\gamma^{\delta\alpha^2}
\gamma^{\delta\alpha^3}
\gamma^{\delta\alpha^4}
\gamma^{\delta\alpha^5}
=
\beta^{-36}$
because with $a = \alpha$, $b = \beta$, $c = \gamma$ and $d = \delta$ this equality contradicts $|c^d c^{da} c^{da^{2}} c^{da^{3}} c^{da^{4}}  c^{da^{5}}| < |b^{12}|$ and all the assumptions of Lemma \ref{lem:gen} apart from $|\alpha| < |\beta|$ are true.

Let $\varepsilon:= \gamma^\delta \gamma^{\delta\alpha} \gamma^{\delta\alpha^2} \gamma^{\delta\alpha^3} \gamma^{\delta\alpha^4} \gamma^{\delta\alpha^5}$. For each integer $n$ and each real number $y$ note
\begin{itemize}
\item $(n,y)\gamma^\delta = (n,y+3)$,
\item $(n+(1/6),y)\gamma^\delta = (n+(1/6),y-1)$,
\item $(n+(1/2),y)\gamma^\delta = (n+(1/2),y-3)$, and
\item $(n+(5/6),y)\gamma^\delta = (n+(5/6),y-1)$.
\end{itemize}
This accounts for all the non-differentiable points of $\gamma^\delta$. In particular all of the non-differentiable points of $\gamma^\delta$ have $x$ coordinate a multiple of $1/6$. Since the set of points with $x$ coordinate a multiple of $1/6$ is setwise stabilised by both $\alpha$ and $\gamma^\delta$ it follows that each non-differentiable point of $\varepsilon$ has $x$ coordinate a multiple of $1/6$. Consequently it is sufficient to show that $\varepsilon$ agrees with $\beta^{-36}$ on points with $x$ coordinate a multiple of $1/6$.

Since $\gamma^{\delta\alpha^6} = \gamma^\delta$ it follows conjugation by $\alpha$ permutes the set of conjugates appearing in the definition of $\varepsilon$. The conjugates appearing in the definition of $\varepsilon$ commute so $\alpha$ commutes with $\varepsilon$. Therefore it is sufficient to check $\varepsilon$ agrees with $\beta^{-36}$ on the points with $x$ coordinate $0$.

For $y$ a real number we have
\begin{itemize}
\item $(0,y) \gamma^\delta = (0,y+3)$,
\item $(0,y) \gamma^{\delta\alpha} = (0,y-1)$,
\item $(0,y) \gamma^{\delta\alpha^2} = (0,y-2)$,
\item $(0,y) \gamma^{\delta\alpha^3} = (0,y-3)$,
\item $(0,y) \gamma^{\delta\alpha^4} = (0,y-2)$, and
\item $(0,y) \gamma^{\delta\alpha^5} = (0,y-1)$.
\end{itemize}
Consequently, for $y$ still a real number,
\begin{align*}
(0,y)\varepsilon &= (0,y+3-1-2-3-2-1) \label{line:sum} \\
              &= (0,y-6)\\
              &= (0,y)\beta^{-36}
\end{align*}
and $\varepsilon = \beta^{-36}$, as desired.
\end{proof}

Fix $\eta:\Rr^2 \to \Rr^2$ defined by the rule $\eta:(x,y) \mapsto (y,x)$. The homeomorphism $\eta$ is not in the image of $f$, but conjugation by $\eta$ is an automorphism of the image of $f$. In particular $\gamma^\eta$ and $\delta^\eta$ are in the image of $f$. Note also that $\alpha^\eta = \beta$ and that $\eta$ is an involution. Let $H:= \An{\alpha,\beta,\gamma,\delta,\gamma^\eta,\delta^\eta}$. Note conjugation by $\eta$ is an automorphism of $H$.

\begin{thm} \label{thm:main}
The group $H$ does not admit a left-order.
\end{thm}
\begin{proof}
By Lemma \ref{lem:CM} we have $\beta \abbor \alpha$ in $\An{\alpha,\beta,\gamma,\delta}$ and therefore also in $H$. Since conjugation by $\eta$ is an automorphism of $H$ we also have $\alpha = \beta^\eta \abbor \alpha^\eta = \beta$ in $H$. Since no left-order can have $\alpha < \beta$ and $\beta < \alpha$ no left-order on $H$ can exist.

\end{proof}



\begin{thebibliography}{00}

\bibitem{Calegari}
\newblock{D. Calegari,}
\newblock{Orderability, and groups of homeomorphisms of the disk.}
\newblock{{\em{Geometry and the imagination}} [blog].}
\newblock{https://lamington.wordpress.com/2009/07/04/orderability-and-groups-of-homeomorphisms-of-the-disk/}

\bibitem{Navas1}
\newblock{B. Deroin, A. Navas, C. Rivas,}
\newblock{Groups, Orders, and Dynamics.}
arXiv:1408.5805v2

\bibitem{CalegariRolfsen}
\newblock{D. Calegari, D. Rolfsen,}
\newblock{Groups of PL homeomorphisms of cubes.}
\newblock{\em{Ann. Math. Toulouse}}
\newblock{Volume $24$,  Number $5$, ($2015$), p. $1261-1292$}

\bibitem{ClayRolfsen}
\newblock{A. Clay, D. Rolfsen,}
\newblock{Ordered Groups and Topology}.
\newblock{\em{American Mathematical Soc.}}
\newblock{$2016$, ISBN 1470431068, 9781470431068}

\bibitem{Navas2}
\newblock{A. Navas,}
\newblock{Group actions on 1-manifolds: a list of very concrete open questions.}
\newblock{\em{Proceedings of the ICM}}
\newblock{(2018) arXiv:1712.06462}
    
\bibitem{Kourovka}
\newblock{Unsolved Problems in Group Theory. The Kourovka Notebook. No. 18 (English version)}
\newblock{edited by V. D. Mazurov, E. I. Khukhro.}
arXiv:1401.0300v14
\end{thebibliography}
\end{document}